\documentclass{amsart}
\usepackage{hyperref}
\hypersetup{
    colorlinks=true,
    linkcolor=blue,
    filecolor=magenta,      
    urlcolor=cyan,
}

\usepackage[table]{xcolor} 
\usepackage[utf8]{inputenc}
\usepackage[english]{babel}
\usepackage{amssymb}
\usepackage{amsthm}
\usepackage{todonotes}

\numberwithin{equation}{section}

\newtheorem{proposition}{Proposition}[section]
\newtheorem{lemma}[proposition]{Lemma}
\newtheorem{theorem}[proposition]{Theorem}
\newtheorem*{theorem*}{Theorem}
\newtheorem{corollary}[proposition]{Corollary}
\theoremstyle{definition}

\title{A note on isomorphisms between submonoids of $\mathbb N^k$ and numerical semigroups}

\author{Jerson Borja}

\date{\today}

\newcommand{\Addresses}{{
  \bigskip
  \footnotesize

J. Borja: \textsc{}\par\nopagebreak
  \textit{E-mail address}: \texttt{jersonborjas@correo.unicordoba.edu.co\\
  Departamento de Matemáticas y Estadística\\
Universidad de Córdoba-Colombia}
}}

\begin{document}
\maketitle

\begin{abstract}
    Given a submonoid $H$ of $\mathbb N^k$, we give some characterizations of the minimum $r\in \mathbb N^+$ such that $H$ is isomorphic to a submonoid of $\mathbb N^r$. In the context of submonoids of $\mathbb N$, we prove that if two numerical semigroups are isomorphic submonoids of $\mathbb N$, then they are equal and the identity map is the unique isomorphism between them. 
\end{abstract}

\section{Introduction}

The symbol $\mathbb N$ stands for the set $\{0,1,2,\ldots,\}$. Elements in $\mathbb N^k$ will be regarded as functions $h:\{1,2,\ldots, k\}\to \mathbb N$. Sometimes it is more convenient to see $h\in \mathbb N^k$ as the $k$-tuple $(h(1), h(2), \ldots, h(k))$, and we also write $h=(h(1), h(2), \ldots, h(k))$ for simplicity.

The set $\mathbb N^k$ with the usual sum is a monoid with identity element $(0,0,\ldots,0)$ that we will denote by $0_k$.

For simplicity we will call a submonoid of $\mathbb N^k$ a \textit{$k$-monoid}.

One interesting property of every $k$-monoid $H$ is that $H$ has a minimal generating set that we denote $\beta(H)$. This set $\beta(H)$ has the property that if $X\subseteq H$ generates $H$, then $\beta(H)\subseteq X$. Moreover, $\beta(H)=H^*\setminus (H^*+H^*)$, where $H^*=H\setminus \{0_k\}$.

When $H$ is finitely generated, we can define the \textit{dimension} of $H$, $dim\ H$, as the cardinality of the set $\beta(H)$, $dim\ H=|\beta(H)|$.

If $H$ is a $k$-monoid and $F$ is an $l$-monoid, then a function $\varphi:H\to F$ is \textit{additive} if it satisfies $\varphi(h_1+h_2)=\varphi(h_1)+\varphi(h_2)$ for all $h_1, h_2\in H$. If $\varphi:H\to F$ is additive, then $\varphi(0_k)=0_l$. An ismomorphism between $H$ and  $F$ is a bijective additive function $\varphi:H\to F$.

A result of J. C. Rosales \cite{rosalesarticle} establishes that a finitely generated commutative monoid is isomorphic to $k$-monoid if and only if it is cancellative, torsion free and has no other units than zero. There exist $k$-monoids that are not finitely generated. For instance, for $k>1$, the set $\{h\in \mathbb N^k: 1\leq h(1)\leq h(2)\leq \cdots \leq h(k)\}\cup \{0_k\}$ is a $k$-monoid that is not finitely generated. If $k=1$, every 1-monoid is finitely generated, and furthermore every nontrivial 1-monoid is isomorphic to a \textit{numerical semigroup} (see \cite{rosalesbook}).

To state our main results we make some comments and introduce terminology.

Let $H$ be a $k$-monoid. Then for every $r>k$ there exists an $r$-monoid that is isomorphic to $H$. In fact, for any $r\in \mathbb N^+$ consider the map $\iota_r:\mathbb N^r\to \mathbb N^{r+1}$ given by $\iota_r(h)=(0,h(1),\ldots, h(r))$ for all $h\in \mathbb N^r$. This map $\iota_r$ is a monomorphism. So, if $r>k$, then $H\cong \iota_{r-1}(\cdots\iota_k(H)\cdots)$. Thus, the question of interest is to determine when there exists an $r$-monoid isomorphic to $H$ if $r<k$.

We associate to $H$ the set $\mathcal I(H):=\{r\in \mathbb N^+: H\cong K\text{ for some $r$-monoid }K\}$. We define the \textit{index} of $H$, $ind\ H$, as the minimum element of $\mathcal I(H)$:
\begin{equation*}
    ind\ H=\min \{r\in \mathbb N^+: H\cong K\text{ for some $r$-monoid }K\}.
\end{equation*}

Given a nonempty subset $A$ of $\mathbb N^k$, let $\mathbb Q^A$ be the $\mathbb Q$-vector space of all functions $\alpha:A\to \mathbb Q$ with the usual sum and scalar product. For each $i\in I_k$ we define $\alpha_i^A:A\to \mathbb Q$ by $\alpha_i^A(f)=f(i)$ for all $f\in A$. We will denote by $vect\ A$ the subspace of $\mathbb Q^A$ spanned by the set $\{\alpha_i^A:i\in I_k\}$.

For any $k\in \mathbb N^+$, let $I_k=\{1,2,\ldots, k\}$. Let $A$ be a nonempty subset of $\mathbb N^k$. We say a nonempty subset $X$ of $I_k$ is \textit{independent over} $A$ if $\{\alpha_i^A:i\in X\}$ is a linearly independent subset of $\mathbb Q^A$. If $\{\alpha_i^A:i\in X\}$ is linearly dependent in $\mathbb Q^A$, we say that $X$ is \textit{dependent over} $A$.

We easily see that if $X\subseteq I_k$ is independent over $A$ and $\varnothing \neq Y\subseteq X$, then $Y$ is independent over $A$; also, if $X\neq \varnothing$ is dependent over $A$ and $X\subseteq Y\subseteq I_k$, then $Y$ is dependent over $A$. Besides, if $\varnothing \neq B\subseteq A\subseteq \mathbb N^k$ and $X$ is dependent over $A$, then $X$ is dependent over $B$, or equivalently, if $X$ is independent over $B$, then $X$ is independent over $A$.

We say that $X\subseteq I_k$ is \textit{maximal independent over} $A$ if $X$ is independent over $A$ and there does not exist $Y\subseteq I_k$ independent over $A$ such that $X\subsetneq Y$. In other words, $X$ is maximal independent over $A$ if and only if $\{\alpha_i^A:i\in X\}$ is a basis of $vect\ A$. Note that if $A\neq \{0_k\}$ is nonempty, then there exists some nonempty set $X\subseteq I_k$ that is maximal independent over $A$, since there is some basis of $vect\ A$ contained in $\{\alpha_i^A:i\in I_k\}$.

For a $k$-monoid $H$ we define the \textit{index of free submonoids of} $H$, denoted $free\ H$, as the maximal dimension of a free submonoid of $H$:   
\begin{equation*}
free\ H=max\{dim\ F: F\subseteq H\text{ and $F$ is free}\}.    
\end{equation*}

One of our main results is that for any nontrivial $k$-monoid $H$, the equality
\begin{equation}
    ind\ H=|X|=free\ H,
\end{equation}
holds, where $X\subseteq I_k$ is maximal independent over $H$.

When $H$ is a finitely generated $k$-monoid and $\beta(H)=\{h_1, h_2,\ldots, h_r\}$ where $r=dim\ H$, we associate to $H$ a $k\times r$ matrix $M^H$ that we will call the \textit{matrix of $H$} (that depends on the order of the elements in $\beta(H)$) given by  
\begin{equation}
    M^H=\begin{pmatrix}
    h_1(1)& h_2(1) &\cdots &h_r(1)\\
    h_1(2)& h_2(2) &\cdots &h_r(2)\\
    \vdots& \vdots &\ddots &\vdots\\
    h_1(k)& h_2(k) &\cdots &h_r(k)
    \end{pmatrix}
\end{equation}

In case $H$ is a finitely generated $k$-monoid, we also prove that 
\begin{equation*}
ind\ H=rank\ M^H,    
\end{equation*}
where $rank\ M^H$ is the rank of the matrix $M^H$.

We will give some characterizations of $k$-monoids of index 1. Here, we come to the ambit of numerical semigroups. A numerical semigroup is a 1-monoid $H$ such that $\mathbb N\setminus H$ is finite. This condition is equivalent to say that $H$ is nontrivial and $\gcd H=1$. Our main result in this context is that if two numerical semigroups are isomorphic, then they are actually equal. This means that the relation of isomorphism on numerical semigroups is trivial.

\section{Index of $k$-monoids}

The following basic properties of $ind\ H$ are easy to prove.

\begin{lemma}
\label{lmindex}
The following statements are true.
\begin{enumerate}
    \item If $H$ is $k$-monoid, then $ind\ H\leq k$.
    \item $ind\ \mathbb N^k=k$.
    \item\label{itemindH} $\mathcal I(H)=\{r\in\mathbb Z^+: r\geq ind\ H\}$.
    \item If $H$ and $F$ are $k$-monoids and $F\subseteq H$, then $ind\ F\leq ind\ H$. 
    \item For any $f\in \mathbb N^k$, $ind\ \langle f\rangle =1$.
    \item If $H\cong F$, then $ind\ H=ind\ F$.
\end{enumerate}
\end{lemma}

Property (\ref{itemindH}) in Lemma \ref{lmindex} means that if $H$ is a $k$-monoid and $r\geq ind\ H$, then we can find an $r$-monoid $F$ isomorphic to $H$, but if $r<ind\ H$, there does not exist such $r$-monoid $F$. Thus, we can find $r$-monoids isomorphic to $H$ for some $r<k$ only if $ind\ H<k$. We will see how under certain conditions on $H$ and $1\leq r<k$ we can construct an $r$-monoid isomorphic to $H$ (Proposition \ref{propH|X} below).

\begin{proposition}
\label{propindexfree}
If $H$ is a free and nontrivial $k$-monoid, then $ind\ H=dim\ H$.
\end{proposition}

\begin{proof}
If $r=dim\ H$, then $H\cong \mathbb N^r$ and therefore $ind\ H=ind\ \mathbb N^r=r=dim\ H$. 
\end{proof}

Let $I_k=\{1,2,\ldots,k\}$. For a subset $X=\{a_1,a_2,\ldots, a_r\}$ of $I_k$, where $a_1<a_2<\cdots<a_r$, let $\eta_X:I_{r}\to I_k$ be defined by $\eta_X(j)=a_j$, for $j\in I_{r}$. For $h\in \mathbb N^k$ we define $h|X:=h\circ \eta_X\in \mathbb N^{r}$; for $A\subseteq \mathbb N^k$ we define $A|X:=\{h|X:h\in A\}\subseteq \mathbb N^r$.

\begin{proposition}
\label{propH|X}
Let $H$ be a submonoid of $\mathbb N^k$ and $X\subseteq I_k$ nonempty. Assume that for each $j\in I_k\setminus X$, $\alpha_j^H\in span\ \{\alpha_i^H:i\in X\}$. Then the map $\phi:H\to H|X$ given by $\phi(h)=h|X$ for all $h\in H$ is an isomorphism.
\end{proposition}

\begin{proof}
It is easy to see that $\phi$ is additive and surjective. By hypothesis, if $j\in I_k\setminus X$, then for each $i\in X$ there exists $c_{i,j}\in \mathbb Q$ such that $h(j)=\sum_{i\in X}c_{i,j}h(i)$ for all $h\in H$. If $g|X=h|X$ for $g,h\in H$, then $g(i)=h(i)$ for all $i\in X$ and then, for all $j\in I_k\setminus X$, $g(j)=\sum_{i\in X}c_{i,j}g(i)=\sum_{i\in X}c_{i,j}h(i)=h(j)$. Thus $g=h$. This shows that $\phi$ is injective and therefore $\phi$ is an isomorphism.
\end{proof}

\begin{proposition}
\label{propindependentbeta(H)}
Let $H$ be a nontrivial $k$-monoid and $X\subseteq I_k$ nonempty. Then $X$ is independent over $H$ if and only if $X$ is independent over $\beta(H)$.
\end{proposition}

\begin{proof}
Since $\beta(H)\subseteq H$, if $X$ is independent over $\beta(H)$, then $X$ is also independent over $H$. Now, if $X$ is dependent over $\beta(H)$, then there are rationals numbers $c_i$, $i\in X$, not all zero, such that $\sum_{i\in X}c_i\alpha_i^{\beta(H)}=0$; this means that $\sum_{i\in X}c_ig(i)=0$ for all $g\in\beta(H)$. If $h\in H$, then $h=\sum_{g\in \beta(H)}d_gg$ for some nonnegative integers $d_g$, $g\in \beta(H)$, so 
\[\sum_{i\in X}c_i\alpha_i^{H}(h)=\sum_{i\in X}c_i\left(\sum_{g\in\beta(H)}d_gg(i)\right)=\sum_{g\in\beta(H)}d_g\left(\sum_{i\in X}c_ig(i)\right)=0.\]
Thus, $X$ is dependent over $H$.
\end{proof}

\begin{corollary}
\label{corindependentbeta(H)}
If $H$ is a nontrivial $k$-monoid, then $X\subseteq I_k$ is maximal independent over $H$ if and only if $X$ is maximal independent over $\beta(H)$. 
\end{corollary}

If $H$ is a $k$-monoid, $F$ is an $l$-monoid and $\phi:H\to F$ is an additive map, then we can define a $\mathbb Q$-linear map $\phi^{\ast}:\mathbb Q^F\to \mathbb Q^H$ given by $\phi^{\ast}(\alpha)=\alpha\circ \phi$ for $\alpha\in \mathbb Q^F$.

\begin{proposition}
\label{propindice4}
Let $H$ be a nontrivial $k$-monoid and $X\subseteq I_k$ independent over $H$. If $r=|X|$, then $I_r$ is independent over $H|X$.
\end{proposition}

\begin{proof}
Write $X=\{a_1,a_2,\ldots, a_r\}$ where $a_1<a_2<\cdots<a_r$. If $h\in H$, then $(h|X)(j)=h(a_j)$, that is,  $\alpha_j^{F|X}(h|X)=\alpha_{a_j}^H(h)$ for all $1\leq j\leq r$. If $\phi:H\to H|X$ is the function defined by $\phi(h)=h|X$ for $h\in H$, then we have that $\phi^*(\alpha_j^{H|X})=\alpha_j^{H|X}\circ\phi=\alpha_{a_j}^H$, $j=1,\ldots, r$. 

If $I_r$ were dependent over $H|X$, then $\sum_{j=1}^rc_j\alpha_i^{H|X}=0$, where $c_1,\ldots, c_r\in\mathbb Q$ are not all zero. Hence 
\begin{equation*}
  0=\phi^{\ast}(\sum_{j=1}^rc_j\alpha_j^{H|X})=\sum_{j=1}^rc_j\phi^{\ast}(\alpha_j^{H|X})=\sum_{j=1}^rc_j\alpha_{a_j}^H,  
\end{equation*}
that is, $X$ would be dependent over $H$. 
\end{proof}

\begin{proposition}
\label{propfinitelyFsubseteqH}
Let $H$ be a nontrivial $k$-monoid and $X\subseteq I_k$ maximal independent over $H$.
Then, there exists a finitely generated $k$-monoid $F$ such that $F\subseteq H$ and $X$ is maximal independent over $F$.
\end{proposition}

\begin{proof}
If $H$ is finitely generated, we take $F=H$. Suppose that $H$ is not finitely generated and that $\beta(H)=\{h_1,h_2,\ldots,h_n,\ldots\}$. For each $n\in \mathbb N^+$ let $F_n=\langle h_1,h_2,\ldots,h_n\rangle\subseteq H$. 

For each $n\in \mathbb N^+$, the set $\{\alpha_i^{F_n}:i\in X\}$ spans $vect\ F_n$. In fact, given that $\{\alpha_i^{H}:i\in X\}$ is a basis of $vect\ H$, if $j\in I_k\setminus X$, then $\alpha_j^H=\sum_{i\in X}c_i\alpha_i^H$ for some $c_i\in \mathbb Q$, $i\in X$. Now, since the relation $\alpha_r^H|_{F_n}=\alpha_r^{F_n}$ holds for $1\leq r\leq k$, it results that 
\[\alpha_j^{F_n}=\alpha_j^H|_{F_n}=\sum_{i\in X}c_i\alpha_i^H|_{F_n}=\sum_{i\in X}c_i\alpha_i^{F_n}.\]
This shows that $\{\alpha_i^{F_n}:i\in X\}$ spans $vect\ F_n$. Now, $\{\alpha_i^{F_1}:i\in X\}$ spans $vect\ F_1$ and so there exists $X_1\subseteq X$ such that $\{\alpha_i^{F_1}:i\in X_1\}$ is a basis of $vect\ F_1$, that is, $X_1$ is maximal independent over $F_1$. Since $F_1\subseteq F_2$, $X_1$ is independent over $F_2$, and this means that $\{\alpha_i^{F_2}:i\in X_1\}$ is linearly independent in $vect\ F_2$. Now, $\{\alpha_i^{F_2}:i\in X\}$ spans $vect\ F_2$ and so there exists $X_2\subseteq X$ that contains $X_1$ and is maximal independent over $F_2$. By continuing this way we construct an increasing sequence of subsets of $X$, $X_1\subseteq X_2\subseteq \cdots\subseteq X_n\subseteq \cdots\subseteq X$ in such a way that $X_n$ is maximal independent over $F_n$ for all $n\in \mathbb N^+$. Since $X$ is finite, the sequence $\{X_n\}$ stabilizes and so there exists $N\in \mathbb N^+$ such that $X_n=X_N$ for all $n\geq N$.

We claim that $X_N=X$. Assume, on the contrary that $X_N\neq X$. Let $n\geq N$. Since $\{\alpha_i^{F_n}:i\in X_N\}$ is a basis of $vect\ F_n$, there exist unique $c_{i,n}\in \mathbb Q$, $i\in X_N$, such that $\alpha_j^{F_n}=\sum_{i\in X_N}c_{i,n}\alpha_i^{F_n}$. 

Now, for $n\geq N$, the relations $\alpha_r^{F_{n+1}}|_{F_n}=\alpha_r^{F_n}$, $1\leq r\leq k$, implies that
\[\alpha_j^{F_n}=\alpha_i^{F_{n+1}}|_{F_n}=\sum_{i\in X_N}c_{i,n+1}\alpha_i^{F_{n+1}}|_{F_n}=\sum_{i\in X_N}c_{i,n+1}\alpha_i^{F_n},\]
so that by uniqueness of the $c_{i,n}$, it follows that $c_{i,n}=c_{i,n+1}$ for all $i\in X_N$. In particular, $c_{i,n}=c_{i,N}$ for all $i\in X_N$ and all $n\geq N$.

Let us set $c_i=c_{i,N}$ for each $i\in X_N$. Since $X\neq X_N$, there is some $j\in X\setminus X_N$. For $n\geq N$ we have $\alpha_j^{F_n}=\sum_{i\in X_N}c_i\alpha_i^{F_n}$. Let $h\in H$. Then $h=a_1g_1+\cdots a_tg_t$ where $a_1,\ldots,a_t\in \mathbb N$ and $g_1,\ldots, g_t\in\beta(H)$. Choose $n\geq N$ big enough so that $\{g_1,\ldots, g_t\}\subseteq \{h_1,\ldots, h_n\}\subseteq F_n$. Then $h\in F_n$ and \[\alpha_j^{H}(h)=\alpha_j^{F_n}(h)=\sum_{i\in X_N}c_i\alpha_i^{F_n}(h)=\sum_{i\in X_N}c_i\alpha_i^{H}(h).\]
This shows that $\alpha_j^H=\sum_{i\in X_N}c_i\alpha_i^{H}$, that contradicts the fact that $X$ is independent over $H$. 

Let $F=F_N$. Then $F\subseteq H$ is finitely generated and $X$ is independent over $F$. That $X$ is maximal independent over $F$ follows from the fact that if $X\subseteq Y\subseteq I_k$ and $Y$ is independent over $F$, then, since $F\subseteq H$, $Y$ is also independent over $H$, and thus $Y=X$.
\end{proof}

Now we are ready to prove our main result.

\begin{theorem}
\label{teoindice}
Let $H$ be a nontrivial $k$-monoid and $X\subseteq I_k$ maximal independent over $H$. Then 
\begin{equation*}
  ind\ H=|X|=free\ H.  
\end{equation*}
Moreover, if $H$ is finitely generated, then 
\begin{equation*}
    ind\ H=rank\ M^H.
\end{equation*}
\end{theorem}

\begin{proof}
By Proposition \ref{propH|X}, $H\cong  H|X$ and by Lemma \ref{lmindex}, since $H|X$ is a $|X|$-monoid, $ind\ H=ind\ (H|X)\leq |X|$. Besides, if $F\subseteq H$ is a free $k$-monoid, then by Proposition \ref{propindexfree} and Lemma \ref{lmindex}, $dim\ L=ind\ L\leq ind\ H$, so $free\ H\leq ind\ H$.

To complete the proof, we construct a free $k$-monoid $F\subseteq H$ such that $dim\ F=|X|$. 

By Proposition \ref{propfinitelyFsubseteqH}, there exists a $k$-monoid $L\subseteq H$ such that $X$ is maximal independent over $L$. Consider the matrix $M^L$, whose columns are basically the elements in $\beta(L)$. By Corollary \ref{corindependentbeta(H)}, $X$ is maximal independent over $\beta(L)$. This means that the matrix $M^L$ has rank $|X|$. Since the rank of $M^L$ equals the maximal number of linearly independent columns of $M^L$, we find that there are $|X|$ elements of $\beta(L)$, say $h_1,h_2,\ldots, h_{|X|}$, that are linearly independent in $\mathbb Q^k$.

Let $F=\langle h_1,h_2,\ldots, h_{|X|}\rangle$. Then $F$ is a free $k$-monoid, $F\subseteq L\subseteq H$ and $dim\ F=|X|$. 

Note that in case $H$ is finitely generated, we can take $L=M$ and have therefore $ind\ H=rank\ M^H$.
\end{proof}

\begin{corollary}
\label{corindexindependent}
Let $H$ be a nontrivial $k$-monoid and $X\subseteq I_k$ independent over $H$. Then $ind\ (H|X)=|X|$.
\end{corollary}

\begin{proof}
By Proposition \ref{propindice4}, if $r=|X|$ then $I_r$ is independent over $H|X$, and therefore, $I_r$ is maximal independent over $H|X$. By Theorem \ref{teoindice}, $ind\ (H|X)=|I_r|=r=|X|$.
\end{proof}

\begin{proposition}
\label{propindexfree2} Let $H$ be a nontrivial $k$-monoid. Then $r=ind\ H$ if and only if there exists $B\subseteq \beta(H)$ such that 
\begin{enumerate}
    \item $|B|=r$,
    \item $\langle B\rangle$ is free, and
    \item for any $h\in \beta(H)\setminus B$ there exist $c\in \mathbb N^+$ and $f,g\in \langle B\rangle$ such that $f+ch=g$. 
\end{enumerate}
\end{proposition}

\begin{proof}
If $r=ind\ H$, then, as in the last part of the proof of Theorem \ref{teoindice}, we can take $B=\{h_1,h_2,\ldots, h_r\}$ (we know that $r=|X|$ where $X$ is maximal independent over $H$). Then $|B|=r$ and $\langle B\rangle$ is free. Let $h\in \beta(H)\setminus B$. Since $\langle B\cup \{h\}\rangle$ cannot be free (otherwise $ind\ H\geq r+1$), there are $n, n_1,\cdots,n_r, m, m_1,\ldots, m_r,\in \mathbb N$ such that $\{n, n_1,\ldots, n_r\}\neq \{m, m_1,\ldots, m_r\}$ and 
\begin{equation*}
n_1h_1+\cdots+n_rh_r+nh=m_1h_1+\cdots+m_rh_r+mh.     
\end{equation*}
If $n=m$, then $n_1h_1+\cdots+n_rh_r=m_1h_1+\cdots+m_rh_r$ and $n_1=m_1,\ldots, n_r=m_r$ (since $\langle B\rangle$ is free), that is, $\{n, n_1,\ldots, n_r\}=\{m, m_1,\ldots, m_r\}$, a contradiction. Thus, $n\neq m$, say $m<n$. If $c=n-m$, $f=n_1h_1+\cdots+n_rh_r$ and $g=m_1h_1+\cdots+m_rh_r$, then $c>0$ and $f+ch=g$.

Conversely, if there is some $B\subseteq \beta(H)$ that satisfy the three conditions, then by $(1)$, $(2)$ and Theorem \ref{teoindice}, $r\leq ind\ H$. If $ind\ H>r$, then, using the last part of the proof of Theorem \ref{teoindice}, we can find a subset $C\subseteq \beta(H)$ such that $|C|=ind\ H$. We can form a $k\times (r+ind\ H)$ matrix $M$ whose first $r$ columns are the elements of $B$ and its last $ind\ H$ columns are those elements in $C$. Condition $(3)$ assures that the rank of the matrix $M$ is $r$. Also, there are $ind\ H$ linearly independent columns coming from $C$, so that the rank of $M$ is at least $ind\ H>r$. This is a contradiction. Therefore, $ind\ H=r$.
\end{proof}

For instance, if $H_k:=\{h\in \mathbb N^k: 1\leq h(1)\leq h(2)\leq \cdots \leq h(k)\}\cup \{0_k\}$ then $H_k$ is a $k$-monoid that is not finitely generated. We show that $ind\ H_k=k$ by finding a free submonoid of $H_k$ of dimension $k$. In fact, the $k$ elements $(1,1,1,\ldots,1), (1,2,2,\ldots, 2), (1,2,3\ldots, 3), \ldots, (1,2,3,\ldots, k)\in H_k$ are linearly independent in $\mathbb Q^k$, so they generate a free submonoid of $H_k$ of dimension $k$. Hence $k\leq H_k$, but really $ind\ H_k=k$ because $H_k$ is a $k$-monoid. In particular, $H_k\not\cong H_r$ if $k\neq r$.

\section{$k$-monoids of index 1}

In this section we give characterizations of $k$-monoids of index 1. We start with the following lemma.

\begin{lemma}
\label{lemacf=dg}
Let $f,g\in \mathbb N^k$. If $cf=dg$ for some $c,d\in \mathbb N^+$, then there exist $h\in \mathbb N^k$, $r,s\in \mathbb N^+$ such that $f=rh$ y $g=sh$. In particular, $f,g\in \langle h\rangle$.
\end{lemma}

\begin{proof}
We have that $cf(i)=dg(i)$ for $1\leq i\leq k$. By cancelling common factors we can assume that $c$ and $d$ are relatively prime. Then, for all $1\leq i\leq k$, $f(i)=df'(i)$ and $g(i)=cg'(i)$ for some $f'(i),g'(i)\in\mathbb N$. Then, $cdf'(i)=cf(i)=dg(i)=dcg'(i)$ for all $1\leq i\leq k$, from where $f'(i)=g'(i)$ for all $i$. If $h\in \mathbb N^k$ is given by $h(i)=f'(i)$ for all $i$, then $f=dh$ and $g=ch$.
\end{proof}

\begin{proposition}\label{propindice1Hleqf} 
A $k$-monoid $H$ has index 1 if and only if there exists $f\in \mathbb N^k$ such that $H\subseteq \langle f\rangle$. 
\end{proposition}

\begin{proof} 
If $H\leq \langle f\rangle$, then $1\leq ind\ H\leq ind\ \langle f\rangle=1$, so that $ind\ H=1$. Conversely, suppose that $ind\ H=1$. If $H=\{0_k\}$, then we can take $f=0_k$. 

Assume $H$ is nontrivial. Since $ind\ H=1$, there exists a nontrivial 1-monoid $F$ such that $H\cong F$, but every 1-monoid is finitely generated (see \cite{rosalesbook}), so $H$ is finitely generated. Let $\beta(H)=\{h_1,h_2,\ldots, h_r\}$ where $r=dim\ H$. The matrix $M^H$ associated to $H$ has rank 1 by Theorem \ref{teoindice}. Then, taking $h_1$ and $h_2$, there is a rational $c_1/d_1$ where $c_1,d_1\in \mathbb N^+$ such that $h_2=(c_1/d_1)h_1$, or the same, $c_1h_1=d_1h_2$. By Lemma \ref{lemacf=dg}, there is some $f_1\in \mathbb N^k$ such that $h_1,h_2\in \langle f_1\rangle$. Now, take $h_2$ and $h_3$. Again, there are $c_2,d_2\in \mathbb N^+$ such that $c_2h_2=d_2h_3$; also, there is some $e\in \mathbb N^+$ such that $h_2=ef_1$. Then, $(c_2e)f_1=d_2h_3$. By Lemma \ref{lemacf=dg}, there is some $f_2\in \mathbb N^k$ such that $f_1, h_3\in \langle f_2\rangle$ and therefore, $h_1, h_2, h_3\in \langle f_2\rangle$. We continue in this way until we find $f=f_{r-1}\in \mathbb N^k$ such that $h_1,h_2,\ldots, h_r\in \langle f\rangle$ and hence $H\subseteq \langle f\rangle$.
\end{proof}

The element $f$ such that $H\subseteq \langle f\rangle$ in Proposition \ref{propindice1Hleqf} is not unique in general. We describe a method to determine all such $f$. 

Let us call an element $h\in \mathbb N^k\setminus \{0_k\}$ \textit{primitive} if $\gcd h(I_k)=1$. 

If $h\in \mathbb N^k\setminus\{0_k\}$, then there are unique $c_h\in \mathbb N^+$ and $g_h\in \mathbb N^k$ such that $h=c_hg_h$ and $g_h$ is primitive. In fact, $c_h=\gcd h(I_k)$ and $g_h:I_k\to \mathbb N$ is given by $g_h(i)=h(i)/c_h$ for $i\in I_k$. Let us call $g_h$ the \textit{primitive part} of $h$. 

\begin{proposition}
\label{propindex1primitive} 
Let $H$ be a $k$-monoid generated by $h_1,h_2,\ldots, h_r$ (all nonzero). Then $ind\ H=1$ if and only if $h_1, h_2, \ldots, h_r$ have the same primitive part. 
\end{proposition}

\begin{proof}
If $ind\ H=1$, then by Proposition \ref{propindice1Hleqf}, $H\subseteq\langle f\rangle$ for some $f\in \mathbb N^k$. Of course $f\neq 0_k$. For each $1\leq j\leq r$, $h_j=d_jf$ for some $d_j\in \mathbb N^+$. Then $h_j=(d_jc_f)g_f$ and the primitive part of $h_j$ is $g_f$. Thus, all the $h_j$ have the same primitive part $g_f$.

Conversely, if all $h_j$ have the same primitive part, say it is $f\in \mathbb N^k$, then $H\subseteq \langle f\rangle$ and so $ind\ H=1$ by Propostition \ref{propindice1Hleqf}.
\end{proof}

\begin{proposition}
\label{propuniqueprimitive}
Let $H$ be a nontrivial $k$-monoid such that $ind\ H=1$. Then there is a unique primitive $f\in \mathbb N^k$ such that $H\subseteq \langle f\rangle$. If $g\in \mathbb N^k$ is such that $H\subseteq\langle g\rangle$, then $g=cf$ for some $c\in \mathbb N^+$.
\end{proposition}

\begin{proof}
Assume $H\subseteq\langle f_1\rangle$ and $H\subseteq\langle f_2\rangle$ where $f_1, f_2\in \mathbb N^k$ are primitive. Let $h\in H\setminus \{0_k\}$. Then there are $c_1, c_2\in \mathbb N^+$ such that $h=c_1f_1$ and $h=c_2f_2$. Thus $c_1f_1=c_2f_2$ and by Lemma \ref{lemacf=dg} there are $r,s\in \mathbb N^+$ and $g\in \mathbb N^k$ such that $f_1=rg$ and $f_2=sg$. Since $f_1$ and $f_2$ are primitive, $r=s=1$ and thus $f_1=g=f_2$. This proves uniqueness. 

Now, if $H\subseteq \langle f\rangle$ and $H\subseteq\langle g\rangle$ where $f, g\in \mathbb N^k$ and $f$ is primitive, then by taking $h\in H$, $h\neq 0_k$, we have that $h=c_1f$ and $h=c_2g$ for some $c_1, c_2\in \mathbb N^+$, so that $c_1f=c_2g$ and by Lemma \ref{lemacf=dg}, there are $c, d\in \mathbb N^+$ and $h_0\in \mathbb N^k$ such that $g=ch_0$ and $f=dh_0$. Since $f$ is primitive, $d=1$, so $f=h_0$ and $g=cf$.
\end{proof}

If $H$ is a $k$-monoid generated by $h_1, h_2, \ldots, h_r$ and all the $h_j$ have the same primitive part $f$, then $f$ is the unique such that $H\subseteq \langle f\rangle$ with $f$ primitive. For $1\leq j\leq r$ we have $h_j=d_jf$ for some $d_j\in \mathbb N^+$. The function $\{h_1,h_2,\ldots, h_r\}\to \mathbb N$ that sends $h_j$ to $d_j$ extends additively to a monomorphism $\varphi:H\to \mathbb N$ that gives an isomorphism between $H$ and $\langle d_1, d_2, \ldots, d_r\rangle$. We can take $d=\gcd (d_1, d_2, \ldots, d_r)$ and $H$ is isomorphic to $\langle d_1/d, d_2/d, \ldots, d_r/d\rangle$ (a numerical semigroup). Finally, if $H\subseteq \langle g\rangle$ for some $g\in \mathbb N^k$, then $g=cf$ where $c$ is some factor of $d$. This determines all $f$ in Proposition \ref{propindice1Hleqf}.

\section{Isomorphic numerical semigroups are equal}

Let $H$ be a numerical semigroup. If $H\neq \mathbb N$, then there exists $m_H\in \mathbb N^+$ such that $H$ contains every natural number $n\geq m_H$ and $m_H-1\notin H$. This number $m_H-1$ is known as the \textit{Frobenius number of $H$}. If $H=\mathbb N$ we define $m_H=1$.

If $H$ is a numerical semigroup, then \textit{the multiplicity} of $H$ is the number $h_0=min\ (H\setminus\{0\})$. The set $H\setminus\{h_0\}$ is a numerical semigroup, for $h_0$ cannot be written as the sum of two nonzero elements of $H$.

If $H$ and $K$ are numerical semigroups and $\varphi:H\to K$ is an additive map, then, for $n\in H$ and $m\in \mathbb N$, $\varphi(mn)=m\varphi(n)$, and if $n, m\in H$, then $\varphi(nm)=n\varphi(m)=m\varphi(n)$.

The main result of this section is that if $H$ and $K$ are numerical semigroups and they are isomorphic, then $H=K$ and there is only one \textit{automorphism} of $H$, the identity map.

\begin{lemma}
\label{lemaisoincreasing}
If $H$ and $K$ are isomomorphic numerical semigroups, then there is only one isomorphism $\varphi:H\to K$, that is an increasing function.
\end{lemma}

\begin{proof}
Let $\varphi:H\to K$ be an isomorphism. Then $\varphi(min\ H\setminus \{0\})=min\ K\setminus\{0\}$. In fact, if $h_0=min\ H\setminus\{0\}$, $k_0=min\ K\setminus\{0\}$, $k_1=\varphi(h_0)$ and we assume that $k_1>k_0$, then there is $h_1\in H$ with $h_1>h_0$ such that $\varphi(h_1)=k_0$. Then 
\begin{equation*}
    h_1\cdot k_1=h_1\varphi(h_0)=\varphi(h_1h_0)=h_0\varphi(h_1)=h_0k_0,
\end{equation*}
and since $h_1>h_0$ and $k_1>k_0$, also $h_1k_1>h_0k_0$, a contradiction. Therefore $k_1=k_0$, that is, $\varphi(h_0)=k_0$.

If we write $H=\{0<h_0<h_1<h_2<\cdots\}$ and $K=\{0<k_0<k_1<k_2<\cdots\}$, then $\varphi(h_0)=k_0$. Now, $\varphi|_{H\setminus\{h_0\}}:H\setminus\{h_0\}\to K\setminus\{k_0\}$ is an ismomorphism of numerical semigroups, thus, by a similar argument as above, $\varphi(h_1)=k_1$. And by induction on $n$ we obtain that $\varphi(h_n)=k_n$ for all $n\in\mathbb N$. This shows that $\varphi$ is unique and is an increasing function.
\end{proof}

\begin{lemma}
\label{lemasubnumericalsemigroup}
Suppose that $\varphi:H\to K$ is an isomorphism between numerical semigroups. If $L\subseteq H$ is a numerical semigroup, then $\varphi(L)\subseteq K$ is a numerical semigroup. 
\end{lemma}

\begin{proof}
Since $L$ is a numerical semigroup, $\mathbb N\setminus L$ is finite, and therefore $H\setminus L$ is finite. Then, $\varphi(H\setminus L)=K\setminus \varphi(L)$ is finite and $\mathbb N\setminus \varphi(L)=(\mathbb N\setminus K)\cup (K\setminus \varphi(L))$ is finite. Thus, $\varphi(L)$ is a numerical semigroup. 
\end{proof}

\begin{theorem}
\label{teounicosimple}
If $H$ and $K$ are isomorphic numerical semigroups, then $H=K$. 
\end{theorem}

\begin{proof}
Let $H$ and $K$ be numerical semigroups and let $\varphi:H\to K$ be the unique isomorphism, that by Lemma \ref{lemaisoincreasing} is increasing. 

First we prove that $\varphi(m_H)=m_K$. Let $h_0\in H$ such that $\varphi(h_0)=m_K$. If $k\in K$ is such that $k\geq m_K$, then $k=m_K+r$ for some $r\geq 0$ and there exists $h\in H$ such that $\varphi(h)=k$. Since $\varphi$ is increasing, we have that $h=h_0+s(r)$ for some $s(r)\geq 0$. Then we have 
\begin{equation*}
    h_0\cdot k=h_0\varphi(h)=\varphi(h_0h)=h\varphi(h_0)=(h_0+s(r))\cdot m_K.
\end{equation*}
By replacing $k=m_K+r$ we get $h_0(m_K+r)=(h_0+s(r))m_K$, that yields to $rh_0=s(r)m_K$ and therefore 
\begin{equation*}
  s(r)=\frac{h_0}{m_K}r  
\end{equation*}
for all $r\geq 0$. Let $c=h_0/m_K$. Since $cr=s(r)$ is a natural number for all $r\geq 0$, $c$ is a natural number. Hence $\varphi(h_0+cr)=m_K+r$ for all $r\geq 0$, and so $\varphi^{-1}(m_K+r)=h_0+cr$, which means that 
\begin{equation*}
 \varphi^{-1}(\{m_K,m_K+1,m_K+2,\ldots\})=\{h_0+cr:\ r\geq 0\}.   
\end{equation*}

Now we apply Lemma \ref{lemasubnumericalsemigroup} to the isomorphism $\varphi^{-1}: K\to H$ and the numerical semigroup $L=\{0,m_K,m_K+1,m_K+2,\ldots\}\subseteq K$ to obtain that $\varphi^{-1}(L)=\{0\}\cup \{h_0+cr:\ r\geq 0\}$ is a numerical semigroup. Then, for some $r$ big enough, $h_0+rc$ and $h_0+c(r+1)$ must be consecutive natural numbers, that is, $h_0+c(r+1)=h_0+cr+1$, that yields to $c=1$.

Therefore $\varphi^{-1}(\{m_K,m_K+1,m_K+2,\ldots\})=\{h_0+r:\ r\geq 0\}$. In particular, $\{h_0+r:\ r\geq 0\}\subseteq H$ and so $m_H\leq h_0$. Then, $\varphi(m_H)\leq \varphi(h_0)=m_K$.

By a symmetric argument applied to the isomorphism $\varphi^{-1}:K\to H$ we obtain that $\varphi^{-1}(m_K)\leq m_H$, so that $m_K\leq \varphi(m_H)$. Thus, we have prover that $\varphi(m_H)=m_K$.

Now we show that $m_H=m_K$. In fact, the restriction of $\varphi$ to $M=\{0,m_H,m_H+1,m_H+2,\ldots\}$ gives an isomorphism between $M$ and $L=\{0,m_K,m_K+1,m_K+2,\ldots\}$, so these two numerical semigroups have the same dimension, but $M$ has dimension $m_H$ and $L$ has dimension $m_K$, so that $m_H=m_K$.

By Lemma \ref{lemaisoincreasing} there is only one isomorphism from $M$ to $L=M$, that is indeed the identity map. The restriction $\varphi|_M:M\to M$ is thus the identity of $M$, so for all $h\geq m_H$, $\varphi(h)=h$.

Now we can show that $H=K$. If $h\in H\setminus\{0\}$, then for some $m\in\mathbb N^+$, $mh>m_H$. Then, $m\varphi(h)=\varphi(mh)=mh$, and since $m>0$ it results that $h=\varphi(h)\in K$. This shows that $H\subseteq K$. The analog argument applied to $\varphi^{-1}:K\to H$ shows the inclusion $K\subseteq H$. Thus, $H=K$. 
\end{proof}

Note that by Lemma \ref{lemaisoincreasing}, if $H$ is a numerical semigroup, then the identity map $I:H\to H$ is the only isomorphism that exists. 

\begin{corollary}
Every nontrivial $k$ monoid of index 1 is isomorphic to a unique numerical semigroup. If $H=\langle h_1,h_2,\ldots, h_r\rangle$ is a numerical semigroup, then a $k$-monoid $F$ is isomorphic to $H$ if and only if there exists $f\in \mathbb N^k\setminus\{0_k\}$ such that $F=\langle h_1f, h_2f,\ldots, h_rf\rangle$. 
\end{corollary}

\begin{proof}
The first part of the corollary follows directly from Theorem \ref{teounicosimple}. Let $H=\langle h_1,h_2,\ldots, h_r\rangle$ be a numerical semigroup. Let us assume without loss of generality that $\beta(H)=\{h_1,h_2,\ldots, h_r\}$ and $0<h_1<h_2<\cdots<h_r$. If $F$ is a $k$-monoid isomorphic to $H$, then there is a unique primitive $g\in \mathbb N^k$ and $d_1, d_2,\ldots, d_r\in \mathbb N^+$ such that $d_1<d_2<\cdots<d_r$ and $\beta(F)=\{d_1g,d_2g,\ldots, d_rg \}$. If $d=\gcd \{d_1, d_2, \ldots, d_r\}$, then $F$ is isomomorphic to the numerical semigroup $\langle d_1/d, d_2/d,\ldots, d_r/d\rangle$. By Theorem \ref{teounicosimple}, $h_j=d_j/d$, $j=1,2,\ldots, r$. Thus, $F=\langle(dh_1)g, (dh_2)g, \ldots, (dh_r)g\rangle$ and we can take $f=dg$. The converse is obvious. 
\end{proof}

\Addresses

\end{document}